\documentclass[12pt]{article}

\usepackage[UKenglish]{babel}
\usepackage[margin=2.5cm]{geometry}
\usepackage{amsmath,amsthm,amssymb,latexsym,mathrsfs,amsfonts}
\usepackage{longtable,epsfig,hhline,float,listings,color}
\usepackage{graphicx,caption,subcaption,bm,fancybox,parskip,framed}

\usepackage{comment}

\def\leaderfill{\leaders\hbox to .25em{\hss.\hss}\hfill}

\allowdisplaybreaks

\def\Bbb{\mathbb}
\def\Cal{\mathcal}
\def\Dt{\partial_t}

\def\eb{\varepsilon}
\def\eps{\varepsilon}
\def\bv{\mathbf{v}}

\def\R {\mathbb{R}}

\def\<{\left<}

\def\>{\right>}
\def\Ree{\operatorname{Re}}
\def\Imm{\operatorname{Im}}

\def\Dx{\Delta_x}
\def\tilde{\widetilde}
\def\hat{\widehat}
\def\({\left(}
\def\){\right)}
\def\Res{\operatorname{Res}}

\def\qand{\quad\mbox{and}\quad}
\def\ri{\text{i}}
\def\rd{\text{d}}
\def\re{\text{e}}

\def\fr{\mbox{$\frac{1}{2}$}}
\def\DT{\partial_T}
\def\DX{\partial_X}
\newtheorem{proposition}{Proposition}[section]
\newtheorem{theorem}[proposition]{Theorem}
\newtheorem{corollary}[proposition]{Corollary}

\theoremstyle{definition}

\newtheorem{remark}[proposition]{Remark}

\numberwithin{equation}{section}

\def \no#1#2#3 {{\bf #1} (#3), #2.}
\def \eds#1#2#3 {#1, #2, #3.}

\begin{document}
\begin{center}
  \textbf{\Large Validity of the hyperbolic Whitham}\\[2mm]

    \textbf{\Large modulation equations in Sobolev spaces}
\vspace{.75cm}

\textsf{\large Thomas J. Bridges$^1$, Anna Kostianko$^{1,2}$, {\small and} Sergey Zelik$^{1,2}$}
\vspace{.25cm}

\textit{1. Department of Mathematics, University of Surrey, Guildford GU2 7XH, UK}\footnote{Email: \textcolor{blue}{T.Bridges@surrey.ac.uk},
  \textcolor{blue}{Anna.Kostianko@surrey.ac.uk}, \textcolor{blue}{S.Zelik@surrey.ac.uk}}

\textit{2. School of Mathematics and Statistics, Lanzhou University, Lanzhou 730000 P.R. China}
\vspace{1.0cm}

\setlength{\fboxsep}{10pt}
\doublebox{\parbox{12cm}{
{\bf Abstract.}
It is proved that modulation in time and space of periodic
wave trains, of the defocussing nonlinear Schr\"odinger equation,
can be approximated by solutions of the Whitham modulation equations,
in the hyperbolic case, on a natural time scale.
The error estimates are based on existence, uniqueness, and energy arguments,
in Sobolev spaces on the real line.  An essential part of the proof
is the inclusion of higher-order corrections to Whitham theory,
and concomitant higher-order energy estimates.
}}
\end{center}
\vspace{.15cm}


\section{Introduction}
\setcounter{equation}{0}
\label{sec-intro}

We study the validity of slow modulation of space-time periodic solutions
of the cubic nonlinear Schr\"odinger (NLS) equation
\begin{equation}\label{nls-1}
\ri \partial_t \Psi + \partial_{x}^2
\Psi+ \gamma |\Psi|^2\Psi = 0 \,,
\end{equation}
where $\Psi(t,x)$ is complex valued, $\gamma=\pm1$, $x\in\R$ and $t\geq0$.
This equation
 possesses an exact three-parameter family of plane wave solutions
\begin{equation}\label{nls-wavetrain}
  \Psi(t,x) =  \Psi_0(\omega,k) \re^{\ri\theta}\,,\quad
  \theta=\omega t+kx+\theta_0\,,
\end{equation}
where $\Psi_0$ is real and positive, $\theta_0$ is a constant, and $(\omega,k)$
are the frequency and wave number.  Substitution into (\ref{nls-1})
gives the nonlinear dispersion relation: $\gamma \Psi_0^2=\omega+k^2$.

Modulated solutions are obtained by
introducing the standard geometric optics ansatz:
\begin{equation}\label{0.ans}
  \Psi(t,x) =
  \widetilde\Psi(T,X,\eb):=\widetilde A(T,X,\eb)e^{\ri\eb^{-1}(\omega T+kX+\widetilde\phi(T,X,\eb))},
\end{equation}
where $\eb\ll1$ is a small parameter, $T=\eb t$, $X=\eb x$
and $\widetilde A(T,X,\eb)$ and
$\widetilde \phi(T,X,\eb)$
are the slowly varying amplitude and phase respectively.
Substituting (\ref{0.ans}) into (\ref{nls-1}), separation of real and imaginary
parts, taking the limit $\eps\to0$, and differentiating
the real part with respect to $X$, gives the
\emph{Whitham modulation equations} (WMEs) in the form
\begin{equation}\label{wmes-intro}
\partial_T A + 2(k+ u)\partial_X  A +  A\partial_X u =0 \qand
\partial_T  u +2(k+ u)\partial_X u - 2\gamma  A\partial_X  A = 0\,.
\end{equation}
In these equations
\begin{equation}\label{A-phi-u-def}
A(T,X) = \widetilde{A}(T,X,0)\,,\quad
\phi(T,X)=\widetilde{\phi}(T,X,0)\,,\ \mbox{and}\ u(T,X)
:=\partial_X\phi(T,X)\,.
\end{equation}
The main aim of this paper is to prove that solutions of the WMEs
\eqref{wmes-intro} stay close to the exact solution of (\ref{nls-1}),
when commensurately initialized,
relative to a metric based on the Sobolev space $H^s(\R)$ for some
positive index $s$, when $\eb$ is sufficiently small.

The overarching motivation for this work is validity
of the WMEs in general, starting from an abstract Lagrangian,
or abstract Euler-Lagrange equation.
This target is currently intractable, and progress to date has been
achieved by proving validity of reduction to the WMEs
  for specific equations. Remarkably,
  validity proofs for the WMEs are rare and the only proofs known
  to the authors are validity of the reduction from PDEs of
  Korteweg-de Vries (KdV) type (e.g. \textsc{Bronski et al.}~\cite{bj10,bhj16}
  and references therein)
  and validity of the reduction of the NLS equation to the WMEs
  (e.g.\ \textsc{D\"ull \& Schneider}~\cite{ds09}).
  In the same spirit as validity, the paper
  \textsc{Benzoni-Gavage et al.}~\cite{bnr14} proves,
  for Hamiltonian PDEs of KdV type,
  the link between hyperbolicity of the WMEs and the spectral stability
  of periodic traveling waves to sideband perturbations.

  In this paper the strategy is to start with a fairly simple PDE,
  the cubic NLS equation, but give comprehensive results on how solutions
  of the WMEs
  \eqref{wmes-intro} stay close to the exact solution of (\ref{nls-1})
  in Sobolev spaces.
  Local existence and uniqueness of solutions of both (\ref{nls-1}) and
(\ref{wmes-intro}) is relatively straightforward (the latter given in
\S\ref{sec-swes}, the former a consequence of the proof in
\S\ref{sec-exactsolutions}).  Indeed, with $A=\sqrt{h}$ and
$\gamma=-1$ the equations (\ref{wmes-intro}) are just a variant
of the classical shallow water equations (SWEs).  However, in our
study of (\ref{wmes-intro}) we introduce
higher-order energies which feed into a proof of validity of higher-order
Whitham theory in \S\ref{sec-approx-solns}.  Subtleties do arise
when we start comparing the solutions of (\ref{wmes-intro}) to solutions
of (\ref{nls-1}).
The most surprising of which is that it will be essential
to include higher-order Whitham theory in the approximation
\begin{equation}\label{ur-series}
  \widehat A(T,X,\eps)  =
  A(T,X) + \eb^2A_1(T,X) + \cdots + \eb^{2n}A_n(T,X)\,,
\end{equation}
for some fixed and finite $n$,
with similar expansions for the other variables.
These series may not be convergent, but only estimates for the
existence of a fixed and finite sum of these terms is required.
Indeed $n=1$ is sufficient, but our proof is stated in
terms of arbitrary but finite $n$.  The study
of expansions of the form (\ref{ur-series}) is given in
\S\ref{sec-approx-solns}.
Other idiosyncrasies are recorded in \S\ref{sec-validity-setup}.
Before stating the main result, we review the related literature
on analysis of NLS, from the WKB, semi-classical and integrability
perspectives, and emphasize the new features that arise when
we segue into the specifics of the WMEs validity problem.
\par

There is an extensive literature on the rigorous analysis of
approximate solutions
of NLS \eqref{nls-1} from the perspective of WKB theory and
semi-classical analysis.
The most studied case is for vanishing basic state, where
$\omega=k=0$ and $\Psi_0=0$, and $s$ can be taken to be zero.
This case is referred as
the {\it semiclassical limit} or
{\it supercritical nonlinear optics limit} and is usually studied
    by expanding the functions $\widetilde A$ and
    $\widetilde\phi$ in Taylor series in $\eb$ similar to
    \eqref{ur-series},
    the so-called
    WKB expansions (see e.g. \textsc{Carles}~\cite{carles-book}
  and references therein). In the defocussing case $\gamma=-1$ the leading terms in these
  expansions are related to the compressible Euler equation
  (and the SWEs in the 1D case) and, since the local solvability of these equations in Sobolev spaces is well-understood (e.g.\ \textsc{Majda}~\cite{Majda}),
  the recurrent  equations for the WKB expansions can be solved
  (also locally) in Sobolev spaces and the
  corresponding error estimates deduced
  (see \cite{carles-book} for details).  In the case $\gamma=+1$,
  the leading order equations are ill-posed and the WKB expansions are
  more delicate. In this case the WKB expansions can still be effectively
  studied by working in spaces of analytic functions
  (e.g.\ \textsc{Gerard}~\cite{Ger}).
  \par
  Since the NLS equation (\ref{nls-1}) is integrable, by the Zakharov-Shabat
  formalism, there is  another body of work that
  incorporates this structure into the analysis, particularly
  in the study of the
  semiclassical approximation of the defocussing NLS equation
(e.g.\ Chapter 5 of \textsc{Kamchatnov}~\cite{k00},
\textsc{Jin, Levermore, \& McLaughlin}~\cite{jlm99}, \textsc{Carles}~\cite{carles-book}, \textsc{Grebert \& T. Kappeler}~\cite{gk14}
and references therein).
In the proofs in this paper integrability of the NLS equation is not
used in any way.
The main disadvantage being that our
results do not give any information about the
fine detail of the approximation error, and the structure of the oscillations
that arise in the $\eb\to0$ limit.
The main advantage is that the methodology extends
to non-integrable systems.
\par
Modulation of the general case where the
basic state is non-trivial, $\Psi_0=\Psi_0(\omega,k)\ne0$, which is
the principal case of interest in validity of WMEs,
is much less studied from the WKB perspective.
Although the WKB analysis
 can be formally performed exactly as before, some new difficulties arise due to the fact that
 the amplitude $\Psi_0$ is not in $L_2(\R)$. This difficulty, and
 its appearance in the validity theory, is
 addressed in \S\ref{sec-validity-setup}.

 Another, probably more essential difficulty comes from the fact that
 the underlying wavetrain
   \eqref{nls-wavetrain} may a priori be linearly unstable (which
   is an essential feature of the self-focussing case).
   Unfortunately,
    the longtime validity in Sobolev
    spaces can not be achieved if the corresponding wavetrain
    is spectrally unstable.

    The validity problem in this case (that is, $\Psi_0\neq0$ and
    $\gamma=\pm1$)
    has been studied by \textsc{D\"ull \& Schneider}~\cite{ds09}
    in the framework of functions which are analytic in a strip
    about the real axis (so called Gevrey spaces). In this case,
the stability/instability properties of the underlying wavetrain are not
as important due to
the very fast decay of higher Fourier modes provided by the analyticity,
so that analysis works in both the focussing and defocussing cases.
Validity is proved in \cite{ds09}, locally in time, in Gevrey spaces.
 \par
 However, as pointed out in \cite{ds09}, the analysis in Gevrey spaces
 is not entirely satisfactory from a validity perspective, due
 to the required constraint on the time interval of existence.
 Indeed, it is usually expected that the WKB approximation will
 work at least until blow up of
 the corresponding smooth solution of the limit equations (= formation of caustics in the terminology
  of geometric optics), but it is not clear how to obtain this in the framework of analytic functions.
  Indeed, the standard technique
  is to work with time-dependent Gevrey spaces of
  analytic functions in a strip, in the complex plane about the
  real axis, with  shrinking size (linearly in time) where
  the rate of shrinking
  is determined by the norms of the initial data. For this reason, the
  width of this
  strip for the exact solution typically shrinks to zero fast,
  no matter how long the lifespan
      of the limit solution is, see \cite{ds09}.  As a result, the limit solution to the initial NLS equation can only be justified on a very small
      time interval (in slow time), no matter how big
      its lifespan actually is.

      This latter shortcoming motivates us to study the
      validity of the Whitham approximations for NLS
      in Sobolev spaces where this problem can potentially be
      overcome.  Indeed, the validity of these approximations
    in the defocussing case in Sobolev spaces is stated as
    an open problem  in \textsc{D\"ull \& Schneider}~\cite{ds09}.
    Only the defocussing case is suitable
    for the framework of Sobolev spaces, so we
      predominantly restrict attention henceforth to the case $\gamma=-1$.
    \par

        An outline of the paper is as follows.  First, before
proceeding with the proof of validity, we look more closely
    at the derivation and properties of the WMEs (\ref{wmes-intro})
    in \S\ref{sec-wmes}.
    The conventional view of Whitham theory
    based on an averaged Lagrangian is recorded,
    and then a rigorous derivation of the averaging and approximation
    process is given from first principles,
    which may have independent interest.

    The validity proof starts in \S\ref{sec-swes} with a proof
    of the local solvability of the WMEs
 \eqref{wmes-intro} in Sobolev spaces. The key observation here is its equivalence to the SWEs for which the local
 well-posedness is well understood. We utilize the shallow water
 energy and its natural
   analogues in higher order Sobolev spaces in order to get the desired result.
   \par

   At the next step  we verify in \S\ref{sec-approx-solns}
   the solvability of the recurrent
 equations for the higher order approximations to Whitham
 theory (\ref{ur-series}).
 Since these equations are linear at leading order and
 have the same structure as the linearised SWEs, we can construct
 these expansions using the same energy technique,
  already developed in \S\ref{sec-swes}.
 It is this construction that enables estimates
 for the residual terms to be pushed up to order $\eb^{2n}$
 in the theorem below.
  \par
  In \S\ref{sec-exactsolutions}
  we tie all the results together by proving rigorous estimates
on the comparison between the exact solution $\widetilde\Psi$,
  \begin{equation}\label{exact-solution-nls}
\widetilde\Psi(T,X,\eb)=\widehat\Psi(T,X,\eb)(1+W(T,X,\eb))\,,
  \end{equation}
  and the approximate solutions $\widehat\Psi$ in (\ref{0.ans}),
  of the NLS equation (\ref{nls-1}). In (\ref{exact-solution-nls}),
  the complex-valued function $W$
  (which measures the deviation of $\widehat\Psi$
  from $\widetilde\Psi$) satisfies an
  exact, but singularly perturbed, version of the NLS equation,
  and it is derived in \S\ref{sec-exactsolutions}. We prove
  the smallness of $W$
  using a special energy type estimate inspired by the proof
  of the spectral stability of the underlying wavetrain.
  The main result of the paper, stated here, is proved in
  \S\ref{sec-exactsolutions}.
  \vspace{.15cm}

 \begin{theorem}\label{Th0.main} Let $n\ge1$ in (\ref{ur-series})
 and fixed, and suppose the initial data for the NLS equation
 \eqref{nls-1} with $\gamma=-1$
  has the form
 \begin{equation}
\Psi(0,x)=e^{r_0(\eb x)}e^{\ri(kx+\eb^{-1}\phi_0(\eb x))},\ \ \eb\ll1\,,
 \end{equation}
for some constants $\omega$ and $k$ satisfying $\omega+k^2+1=0$ and some functions $(r_0,\phi_0)$ of the
slow variable $X=\eb x$ satisfying
$$
\|r_0\|_{H^{3n+5}(\R)}+\|\DX\phi_0\|_{H^{3n+5}(\R)}\le C,
$$
where the constant $C$ is independent of $\eb$.
Assume also that the associated WMEs \eqref{wmes-intro} with the
initial data $(r_0,\phi_0)$ possess a smooth solution on some time interval $T\le T_0$ which satisfies
$$
\|r(T)\|_{H^{3n+5}(\R)}+\| u(T)\|_{H^{3n+5}(\R)}\le C,\ \ T\le T_0\,,
$$
for some new constant $C$ also independent of $\eb$,
where $r(T,X)=\ln A(T,X)$.
Then the exact solution $\widetilde \Psi(T,X,\eb)$, of the
NLS equation with the same initial data, exists on a time interval
$T\in[0,T_0]$ and remains $\eb^{2n}$-close
  to the appropriate $n$th order Whitham modulation approximation
  $\widehat \Psi=\widehat Ae^{\ri\eb^{-1}\widehat\Theta}$
   in the following sense:
   \begin{equation}\label{est0-main}
   \|(\widetilde\Psi(T,\cdot)-\widehat\Psi(T,\cdot))e^{-\ri\eb^{-1}\widehat\Theta(T,\cdot)}\|_{H^1(\R)}
   \le C\eb^{2n},\ \ T\le T_0,
   \end{equation}
   where the constant $C$ is independent of $\eb$.
\end{theorem}
 \vspace{.15cm}

 \begin{remark}
   Using the fact that $H^1\subset L^\infty$, and under the assumptions of
 the theorem, we can also compare the solutions $\widehat \Psi$ and $\widetilde\Psi$ without the phase factor  $e^{\ri\eb^{-1}\widehat\Theta}$:
  $$
  \|\widetilde \Psi(T,\cdot)-\widehat\Psi(T,\cdot)\|_{L^\infty(\R)}\le C\eb^{2n},\ \ T\le T_0\,.
  $$
  However, if we want to compare this distance in higher Sobolev norms,
  it will be necessary to
  decrease the order of approximation due to the factor $\eb^{-1}$
  in the phase.
  \end{remark}

  \begin{remark}
  We also would like
  to emphasize that the order of approximation in estimate (\ref{est0-main})
  in Theorem \ref{Th0.main} is $\eb^{2n}$,
  not $\eb^{2(n+1)}$ as one might expect. For this reason, we can not take
  $n=0$ in (\ref{ur-series}) and need to
  retain at least one higher-order term in (\ref{ur-series}).
  It is not clear whether this restriction is
    technical, or something more fundamental is behind it.
    \end{remark}

\section{Variations on Whitham modulation theory}
\label{sec-wmes}
\setcounter{equation}{0}

The modulation ansatz (\ref{0.ans}) is just a geometric
optics approximation.  What makes it Whitham modulation theory (WMT)
is the identification of the first equation in (\ref{wmes-intro}) with
the ``conservation of wave action'' and the
second equation in (\ref{wmes-intro}) with the ``conservation of
waves''.  Also important in Whitham theory is that an averaged
Lagrangian is the organising centre.  It is this triple,
and the generation of conservation of wave action via Noether's Theorem
from the averaged Lagrangian, that gives WMT its universality.
Background on WMT can be
found in \textsc{Whitham}~\cite{w74}, \textsc{Kamchatnov}~\cite{k00},
\textsc{Bridges}~\cite{tjb17}, and references therein.

In this section, we discuss the steps leading up to the WMEs from
first principles
starting with an abstract Lagrangian for a general conservative PDE.
Let $L=L(\bv_t,\bv_x,\bv)$ be a given smooth function of the
vector valued function $\bv(t,x)$ for $t,x\in\R$, and suppose that
the variational principle
\[
\delta\int_{t_1}^{t_2}\int_{x_1}^{x_2}L(\bv_t,\bv_x,\bv)\,\rd x\rd t = 0\,,
\]
with fixed endpoints on the variations $\delta \bv$ generates the
governing equations.
The corresponding Euler-Lagrange equation reads
\begin{equation}\label{eq.EUg}
\partial_t\(L_{\bv_t}(\bv_t,\bv_x,\bv)\)+\partial_x\(L_{\bv_x}(\bv_t,\bv_x,\bv)\)=
L_{\bv}(\bv_t,\bv_x,\bv)\,.
\end{equation}
Now assume that this equation possesses a three parameter family of
single-phase space-time periodic solutions
\begin{equation}\label{eq.wave}
\bv(t,x):=\bv_0(\theta,\omega,k), \ \ \theta=\omega t+kx+\theta_0,\ \ \omega,k,\theta_0\in\R\,,
\end{equation}
where $\bv_0(s,\omega,k)$ is a smooth function, which is $2\pi$-periodic
with respect to $s$
for all $\omega$ and $k$.  Inserting this solution into the Euler-Lagrange
equation,
 gives us a differential equation for $\bv_0(s,\omega,k)$
\begin{multline}\label{EL-idenity}
\(L''_{\bv_t,\bv_t}(\partial_s\bv_0\, \omega,\partial_s\bv_0\, k,\bv_0)\omega^2+
2L''_{\bv_t,\bv_x}(\cdots)\omega k+
L''_{\bv_x,\bv_x}(\cdots)k^2\)\partial_{s}^2\bv_0 +\\+\(L''_{\bv_t,\bv}(\cdots)\omega+
L''_{\bv_x,\bv}(\cdots)k\)\partial_s\bv_0=L'_\bv(\cdots).
\end{multline}
Motivated by geometric optics theory, we
seek for a modulated solution of \eqref{eq.EUg} in the form
\begin{equation}\label{eq.ANg}
\bv(t,x)=\bv_0(\eb^{-1}\Theta(T,X,\eb),\Omega(T,X,\eb),Q(T,X,\eb))+\eb W(T,X,\eb)
\end{equation}
where $\eb\ll1$, $T=\eb t$, $X=\eb x$ are slow variables.  The
unknown functions $\Theta(T,X,\eb)$, $\Omega(T,X,\eb)$ and $Q(T,X,\eb)$ are
slowly-varying functions which are
responsible for evolution along the wavetrain manifold,
and the function $W(T,X,\eb)$ measures
  the evolution in
  transversal directions.  Usually some kind of point-wise orthogonality conditions
  for $\bv_0$ and $W$ are posed in order
   to determine $W$ in a unique way, however, we have $W\equiv0$ in our particular case of NLS, so we prefer
    not to specify any such conditions here.
\par
In order to (formally) get the approximate equations
for quantities $\Theta$, $\Omega$ and $Q$, we substitute the ansatz
\eqref{eq.ANg} into the Lagrangian while dropping the terms of order $\eb$
and higher.
This gives us the truncated Lagrangian
\begin{equation}\label{eq.trunc}
 L_0(s,\Theta_X,\Theta_T,\Omega,Q):=L(\partial_s\bv_0(s,\Omega,Q)\Theta_T,
 \partial_s \bv_0(s,\Omega,Q)\Theta_X,\bv_0(s,\Omega,Q)),\ \ s:=\eb^{-1}\Theta.
\end{equation}
However, this Lagrangian still contains rapidly oscillating terms related
with $\eb^{-1}\Theta$,  so we need to (again formally)
 introduce averaging to get the reduced Lagrangian
\begin{equation}\label{eq.red}
 \Cal L(\Theta_T,\Theta_X, \Omega,Q)=\<L_0(s,\Theta_T,\Theta_X,\Omega,Q)\>_s,
\end{equation}
 where $\<f\>_s=\frac1{2\pi}\int_0^{2\pi} f(s)\,ds$. Taking variational derivatives with
  respect to $\Omega$ and $Q$, we end up with two equations
\begin{equation}\label{eq.OQ}
 \begin{cases}
 \<L'_{\bv_t}\partial_{s\omega}^2\bv_0\>_s\Theta_T+\<L'_{\bv_x}\partial^2_{s\omega}\bv_0\>_s\Theta_X=-
 \<L'_\bv\partial_\omega\bv_0\>_s\\
\<L'_{\bv_t}\partial^2_{sk}\bv_0\>_s\Theta_T+\<L'_{\bv_x}\partial^2_{sk}\bv_0\>_s\Theta_X=-
\<L'_\bv\partial_k\bv_0\>_s
 \end{cases}
\end{equation}
which are satisfied point-wise for all $T,X\in\R$.
This is the linear system with respect to the variables $\Theta_T$ and $\Theta_X$. We claim that
\begin{equation}\label{eq.exact}
\Theta_T=\Omega,\ \ \Theta_X=Q
\end{equation}
solves this system.  A novelty here is that conservation of waves is deduced
from the averaged Lagrangian rather than assuming it \emph{a priori}.

The assertion (\ref{eq.exact})
can be verified by integrating by parts
the integrals in the left-hand side of \eqref{eq.OQ}
(moving the $s$-derivative from the function $\bv_0$ to
$L'_{\bv_t}$ and $L'_{\bv_x}$ and using the identity \eqref{EL-idenity}).
In general, the determinant of the system \eqref{eq.OQ} does not vanish identically,
so the solution \eqref{eq.OQ} is unique.  However, it may be not so in degenerate cases like our NLS example which will be considered in the subsection
below.
   \par

   Inserting \eqref{eq.exact} into the reduced Lagrangian, we finally arrive
   at the
   standard WME averaged Lagrangian
\begin{equation}\label{eq.WMT}
\mathcal L_{\rm WME}(\Omega,Q):=\<L_0(s,\Omega,Q,\Omega,Q)\>_s\,,
\end{equation}
which should be considered under the extra constraint
\begin{equation}\label{eq.con}
\Omega_X=Q_T.
\end{equation}
The associated Euler-Lagrange equations finally give us the desired Whitham
 modulation equations (WMEs) in a standard form
\begin{equation}\label{wave-action}
\partial_T\mathscr{A}(\Omega,Q) +
\partial_X \mathscr{B}(\Omega,Q) =0\,,
\end{equation}
with $\mathscr{A}:=\partial_\Omega\mathscr{L}_{\rm WME}$ and $\mathscr{B}=\partial_Q\mathscr{L}_{\rm WME}$.
The equation (\ref{wave-action}) is called ``conservation of wave action''
in \cite{w74}. Of course, this equation should be considered together with \eqref{eq.con} which gives
 a closed system of two quasi-linear first order equations for determining the wave
  frequency $\Omega$ and wave number $Q$. The wave phase $\Theta$
 is determined after that from the relation
\begin{equation}\label{eq.EF}
d\Theta=\Omega\,dT+Q\,dX
\end{equation}
and the exactness of this differential form is guaranteed by \eqref{eq.con}. Using \eqref{eq.exact}
 these equations can be also rewritten in terms of a single second order quasi-linear
 PDE for the phase $\Theta$:
\begin{equation}\label{theta-action}
\partial_T\mathscr{A}(\Theta_T,\Theta_X) +
\partial_X \mathscr{B}(\Theta_T,\Theta_X) =0\,.
\end{equation}
\begin{remark} There is an effective machinery which allows us to compute the WME
Lagrangian $\Cal L_{\rm WME}(\omega,k)$, namely, we may just put the
 family of solutions \eqref{eq.wave} where $\omega$ and $k$ are constants
  into the initial Lagrangian $L$ and compute the result as a function of $\theta$, $\omega$ and $k$.
   Performing thereafter averaging with respect to the variable $\theta$, we end up
    exactly with the WME lagrangian $\Cal L_{\rm WME}(\omega,k)$, see \eqref{eq.trunc}, \eqref{eq.red}
    and \eqref{eq.WMT}.  These steps are essentially how
    the scheme is usually presented in the
     literature, see e.g. \cite{w74,ds09,tjb17}.
     However, in the conventional approach the extra constraint
     \eqref{eq.con} is just postulated. For this reason, we have given
     a bit more detailed
     derivation here of the WMEs which includes the emergence of this constraint as well.
     Moreover,
        the precise form \eqref{eq.ANg} that we are seeking for the modulated solution is
        also important for the forthcoming justification of the above
        formal procedures. Finally,
        we have presented here the result for a general Lagrangian $L(\bv_t,\bv_x,\bv)$ since the particular case of
        the NLS Lagrangian is "too degenerate" to see the key features
        of the theory.
\end{remark}

\subsection{WMT for the cubic NLS equation}
\label{subsec-wmt-nls}

Now restrict attention to the cubic NLS equation, which is generated
by the Lagrangian
\begin{equation}\label{L-nls-def}
L = \fr \texttt{i} \big( \overline\Psi \partial_t\Psi-\Psi\overline\partial_t\Psi\big)
- \big|\partial_x\Psi\big|^2 + \fr\gamma |\Psi|^4\,.
\end{equation}
The Euler-Lagrange equation, obtained by taking variations
of the integral of $L$ with fixed endpoint conditions, is the
cubic NLS (\ref{nls-1}).

The family of space-time
  periodic solutions \eqref{eq.wave} is given explicitly by
  \begin{equation}\label{wave-NLS}
    \Psi_0(\theta,\omega,k) := \Psi_0(\omega,k)\re^{\ri\theta}\,,
    \quad\mbox{with}\ \gamma \Psi_0^2(\omega,k)=\omega+k^2\,.
  \end{equation}
  Express the amplitude as $\Psi_0(\omega,k)=
  A(\omega,k)$, as when modulated it will be the amplitude in
  the WMEs.  Then the ansatz \eqref{eq.ANg} for the modulated
  solution reads
  \begin{equation}\label{eq.mod-NLS}
\Psi(t,x)=A(\Omega(T,X),Q(T,X))\re^{\ri \frac1\eb \Theta(T,X)}.
  \end{equation}
In this case we do not have any transversal directions to the wavetrain manifold, so $W\equiv0$.
 This observation does not affect the leading order approximate equations
  (which do not contain $W$ in any case), but it is crucial when higher order approximations
   are considered, see next section.
   \par
The truncated Lagrangian $L_0$ now reads
\begin{equation}
L_0(s,\Theta_T,\Theta_X,\Omega,Q)=
-\frac12\gamma\(A^4(\Theta_T,\Theta_X)-(A^2(\Omega,Q)-A^2(\Theta_T,\Theta_X))^2\)
\end{equation}
and we see two important simplifications:
\par
 1) In contrast to the general case, the truncated
 Lagrangian is independent of $s$, so there is nothing to average here. This also makes
  the validity theory essentially simpler.
  \par
 2) The Lagrangian $L_0=\Cal L$ does not depend on $\Omega$ and $Q$ separately, but
 only on their combination $A^2:=\gamma^{-1}(\Omega+Q^2)$. For this reason, the functions
  $\Omega$ and $Q$ are not uniquely defined, but only their combination $A^2$ is. In particular,
  in this case, the determinant of system \eqref{eq.OQ} vanishes identically, so from these equations
  we get only the relation
  $$
  A^2(\Omega,Q)=A^2(\Theta_T,\Theta_X)\ \Longleftrightarrow\ \Omega+Q^2=\Theta_T+\Theta_X^2,
  $$
but for consistency with the general case, we may {\it define} $\Omega=\Theta_T$ and $Q=\Theta_X$.
\par
The WME Lagrangian now reads
$$
\Cal L_{\rm WME}(\Theta_T,\Theta_X)=-\frac12\gamma\(\Theta_T+\Theta_X^2\)^2
$$
and the corresponding WME has the following form
\begin{equation}\label{eq-WME-NLS}
\DT(\Theta_T+\Theta_X^2)+2\DX(\Theta_X(\Theta_T+\Theta_X^2))=0.
\end{equation}
Remarkably, this equation is independent of $\gamma\ne0$.
Moreover, it is not difficult to see that
it is hyperbolic if $\Theta_T+\Theta_X^2<0$, and elliptic if $\Theta_T+\Theta_X^2>0$. It is also
 straightforward to show that the sign of $\Theta_T+\Theta_X^2$ is preserved under the
  time evolution (at least until the solution is smooth). Keeping in mind that
   $\Theta_T+\Theta_X^2=\gamma A^2$ and the amplitude $A$ is real, we conclude that $\gamma>0$
   corresponds to the elliptic case and in the case $\gamma<0$ the equation is hyperbolic.
   \par
It is natural to rewrite equation \eqref{eq-WME-NLS} with respect to new
 variables
 $$
 h(T,X):=A^2(T,X)=\gamma^{-1}(\Theta_T+\Theta_X^2)
 $$
  and $u(T,X):=\DX\Theta(T,X)$. This gives
 \begin{equation}\label{eq-SWE}
\DT u+\DX(u^2-\gamma h)=0\qand
\DT h+2\DX(uh)=0\,.
 \end{equation}
 In the hyperbolic case $\gamma<0$ this system is the classical
 SWEs (up to scaling $2u\to u$), and we will utilize
 the connection between the hyperbolic WMEs and the classical SWEs
 in our proof of validity, in order to get the local solvability of WMEs in
 Sobolev spaces. In order to handle  the degeneracy of the SWEs energy at
 $h=0$, it is also useful to write the amplitude $A$ in the form $A=e^r$
 which leads to the following equations for $(r,u)$,
\begin{equation}\label{eq.ru}
  \begin{array}{rcl}
    && \partial_T r + 2u\partial_X r + \partial_Xu = 0 \\[2mm]
    && \partial_T u + 2u\partial_X u - 2\gamma\re^{2r}\partial_X r = 0 \,.
  \end{array}
\end{equation}

\section{Validity theory: setup and roadmap}
\label{sec-validity-setup}
\setcounter{equation}{0}

There are several dimensions to the approximation theory.
The backbone is the usual three steps in validity theory:
an existence theory for the original
equation, an existence theory for the reduced equation, and
the evolution of a
measure of the distance between the two (e.g.\ Part IV of
\textsc{Uecker \& Schneider}~\cite{su17}).  In addition,
idiosyncrasies arise
that are particular to the context of the
NLS to WMEs reduction in Sobolev spaces.
\par
The first key question is the choice of
function space.  In validity theory,
our main interest is the case where
the basic state (\ref{nls-wavetrain}) is non-trivial,
$\Psi_0\ne0$, and this solution
is {\it not square integrable} on the real line and so, in
contrast to WKB theory \cite{carles-book}, we cannot assume that the
amplitude $A(T)\in L^2(\R)$.  It
 would be natural to consider $A(T)\in L^\infty(\R)$ or $A(T)$ belonging to some uniformly
 local Sobolev space, but this is problematic since the
 NLS equation is not
  well-posed in such spaces.  As a compromise, we will assume that
  $$
  \lim_{X\to\pm\infty}A(T,X)=\Psi_0,\ \ \gamma \Psi_0^2=\omega+k^2,
  $$
so the amplitude of the modulated solution stabilizes as $X\to\pm\infty$ to
the amplitude of the basic wavetrain. The same assumption will be posed also for the amplitudes $\tilde A(T,X,\eb)$
of exact solutions of the NLS equation as well as their  $n$th order approximations $\widehat A(T,X,\eb)$.
\par

In contrast to \cite{ds09} where the problem has
 been considered in spaces of analytic functions, we cannot treat the elliptic case $\gamma>0$ in
 Sobolev spaces (elliptic equations are usually ill-posed in such spaces),
 so we have to
 assume that $\gamma<0$. Taking into the account that $\omega+k^2\ne0$,
 we may scale the variables $t$, $x$ and $\Psi$ in NLS and assume
 without loss of generality that
\begin{equation}\label{eq-scale}
\gamma=-1,\ \  \omega+k^2+1=0,\ \ \Psi_0=1.
\end{equation}

A valuable simplification that arises in the case
    $\Psi_0\ne0$ is that we may
    naturally separate the amplitude $A$ from the
    singularity at $A=0$ by assuming that
\begin{equation}\label{0.nosing}
  \ln\frac{A(T,\cdot)}{\Psi_0(\omega,k)}\in H^s(\R)\,,\quad
  \mbox{for some index $s>0$}\,.
\end{equation}
We will prove below that if this assumption is satisfied for
$T=0$ it will be preserved at least for small positive time $T\le T_0$,
and the same is
true when we replace $A$ by $\widetilde A$ and $\widehat A$
in (\ref{0.nosing}). {\color{black} This assumption is equivalent to
$$
 A(T,\cdot)-\Psi_0\in H^s(\R).
$$
The fact that the NLS equation is ill posed in $L^\infty(\R)$, as well as in
uniformly local Sobolev spaces, makes the assumption \eqref{0.nosing}
appear to be unavoidable here.}

Another overarching assumption is
  \begin{equation}\label{0.sob}
u(T,\cdot):=\DX\phi(T,\cdot)\in H^s(\R),\ \ T\ge0,
  \end{equation}
  for sufficiently large positive index $s$,
  with similar assumptions for exact solutions
  $(\widetilde A,\widetilde u)$
  and their $n$th order Whitham approximations $(\widehat A,\widehat u)$.

We now discuss {\color{black} assumption \eqref{0.sob} and} the analysis of the phase $\Theta$.
The equations for the amplitude and phase in the geometric
optics ansatz \eqref{0.ans}, with
\begin{equation}\label{4.anz}
\widetilde A(T,X,\eb)=e^{\widetilde r(T,X,\eb)}\,,\quad \widetilde \Theta(T,X,\eb)=\omega T+kX+\widetilde \phi(T,X,\eb)
\end{equation}
\color{black}
are
 \begin{equation}\label{eq-tilde}
\begin{array}{rcl}
  \displaystyle
  \partial_T\widetilde\phi + (k+\partial_X\widetilde\phi)^2-\gamma \re^{2\widetilde r}
  +\omega-\eb^2 (\partial_X^2\widetilde r + (\partial_X\widetilde r)^2) &=& 0\\[2mm]
  \partial_T\widetilde r + 2(k+\partial_X\widetilde\phi)\partial_X\widetilde r
  +\partial_X^2\widetilde\phi  &=& 0\,,
\end{array}
\end{equation}
 From the second
equation of (\ref{eq-tilde}) {\color{black} (and the fact that $\widetilde r,\partial_T\widetilde r\in H^s$)},
 we expect that $\DX^2\widetilde\phi$ should be
square integrable and $\DX\widetilde\phi$
 should be bounded, then {\color{black} differentiating the  first
   equation with respect to $X$,} we get
 that$\partial^2_{TX}{\color{black}\widetilde\phi}$ is square integrable,
 which together with the comparison with the basic
   wave train \eqref{wave-NLS} gives the natural assumption:
   \begin{equation}\label{eq-T}
 \DT\widetilde\phi,\DX\widetilde\phi\in H^s(\R).
   \end{equation}
Note that we do not assume that $\widetilde\phi(T)\in H^s(\R)$, only $\DX\widetilde\phi\in H^s(\R)$.
Actually, in general
 we do not have even that $\widetilde\phi(T)\in L^\infty(\R)$ and it may grow as $X\to\infty$ slightly slower
  than $\sqrt{|X|}$. With a slight abuse of notation, we denote
  $$
  \widetilde u(T,X,\eb):=\DX\widetilde\phi(T,X,\eb).
  $$
  We do not claim that assumptions \eqref{0.nosing} and \eqref{eq-T} are
  the most general for verifying
  the validity in Sobolev spaces {\color{black} (for example,
    the slightly more general assumptions, that
   $\widetilde u\in L^\infty$ and $\DX\widetilde u\in L^2$, also look acceptable, but are more
    difficult to implement since infinite energy solutions of
    the shallow water equations should then be considered)}, they are just convenient and look natural to us. So, from now on we
    assume that these conditions are always satisfied. In particular, they are satisfied for
     the initial data at $T=0$:
     \begin{equation}\label{eq.A}
     \widetilde r(0)=r(0):=\ln \widetilde A(0)\in H^s(\R),\ \
      \widetilde u(0)=u(0):=\DX\widetilde \phi(0)\in H^s(\R)
     \end{equation}
for some sufficiently large $s$.

With this setup, our strategy will be as follows.
We fix some initial data $(r(0),\phi(0))$ satisfying
 \eqref{eq-T} and \eqref{eq.A}
for sufficiently large $s$ and consider the corresponding solution $(r(T),u(T))$
satisfying the WMEs (SWEs) in (\ref{eq.ru}),
on  some interval $T\le T_0$. The local existence and uniqueness  of such a solution is proved in
 \S \ref{sec-swes} although we {\it do not assume} that $T_0$ is small.
 \par
 Then, formally Taylor expanding the geometric optics
 representation in $\eb$,
 in terms of the variables $(\widehat r,\widehat\phi)$,
 we get the recurrent linear equations for the corresponding
 Taylor coefficients.  In \S\ref{sec-approx-solns} it is proved
 that the obtained recurrent equations are uniquely solvable on {\it the same}
   time interval $T\in[0,T_0]$, so these expansions are well-defined.
   Truncating Taylor series
   at the $\eb^{2n}$-term and
   denoting the obtained functions by $\widehat A(T,X,\eb)$ and
   $\widehat \phi(T,X,\eb)$ respectively,
   we represent the approximate modulated solution (\ref{0.ans})
   of the NLS in the form (\ref{4.anz}).
   The amplitude and phase of this approximate solution (for $n$ fixed and
   finite) will satisfy the equations
\begin{equation}\label{eq2.3}
\begin{array}{rcl}
  \displaystyle
  \partial_T\widehat\phi + (k+\partial_X\widehat\phi)^2-\gamma \hat A^2
  +\omega-\eb^2\widehat A^{-1}\partial_X^2\widehat A  &=& \Res_\phi\\[2mm]
  \partial_T\widehat A + 2(k+\partial_X\widehat\phi)\partial_X\widehat A
  +\widehat A\partial_X^2\widehat\phi  &=& \Res_A\,,
\end{array}
\end{equation}
with the residual terms $\Res_\phi$ and $\Res_A$ satisfying
\begin{equation}\label{4.res}
\|\Res_A(T,\cdot)\|_{H^{s'}(\R)}+\|\Res_\phi(T,\cdot)\|_{H^{s'}(\R)}\le C\eb^{2(n+1)}\,,\quad T\le T_0\,,
\end{equation}
for some $s'=s'(n)<s$ along with
\begin{equation}\label{A0-phi0}
  \widehat r(T,X,0) = r(T,X) \qand \widehat\phi(T,X,0)=\phi(T,X)\,,
\end{equation}
see \S \ref{sec-approx-solns} for the details.
The results of \S \ref{sec-approx-solns}
then feed into \S \ref{sec-exactsolutions},
where estimates on the residual between $\widehat\Psi$ and the exact solution
$\widetilde\Psi$ are proved, which combine to complete the proof of
Theorem \ref{Th0.main}.

\section{Existence theory for the WMEs in Sobolev spaces}
\label{sec-swes}
\setcounter{equation}{0}

In this section, we treat the basic WMEs with $\gamma=-1$
from the perspective of the shallow water equations, or more generally as
a quasilinear hyperbolic system.  The starting point is
 \begin{equation}\label{6.WMEL}
 \begin{array}{rcl}
 \partial_T r &=& -\partial_X u- 2 (u+k)\partial_X r\,,\quad r\big|_{T=0}=r_0\\[2mm]
 \partial_T u &=& -\partial_X(u+k)^2-\partial_X(e^{2r})\,,\quad u\big|_{T=0}=u_0\,.
 \end{array}
 \end{equation}
 The arguments that we use are standard in the literature on quasilinear
 hyperbolic systems (see e.g., \cite{Majda}),
 so we will just sketch the proof.

 The system \eqref{6.WMEL} has an exact energy conservation law
with energy
 \begin{equation}\label{energy-def}
   E(r,u) =
   \(e^{2r}u^2+\frac12\(e^{2r}-1\)^2,1\)_{L^2}\quad\mbox{satisfying}\quad
   \frac{d\ }{dT} E(r,u) = 0 \,.
 \end{equation}
 We start with the analysis of the linearised non-homogeneous problem
 associated with \eqref{6.WMEL}
  \begin{equation}\label{6.linest}
\begin{cases}
  \partial_T R=-\partial_X U-2U\partial_X r-2(u+k)\partial_X R +H_{r}(T)\,,\quad
  R\big|_{T=0}=R_0,\\[2mm]
  \partial_T U =-2\partial_X((u+k)U)-2\partial_X(e^{2r}R)+H_{u}(T)\,,\quad
  U\big|_{T=0}=U_0,
\end{cases}
  \end{equation}
where $u(T)$ and $r(T)$ are given smooth functions satisfying
\begin{equation}\label{6.apr}
\|r(T)\|_{H^{s+2}}+\|u(T)\|_{H^{s+2}}\le C\,,\quad T\le T_0.
\end{equation}
We need the following result for this inhomogeneous system.  It will also
be used later for verifying the existence of local solutions for
the non-linear system and in the next section for constructing the
higher-order approximate solutions.
  \vspace{.15cm}

 \begin{proposition}\label{Prop6.main} Suppose that
    the smooth functions $r(T)$ and $u(T)$ satisfy \eqref{6.apr}.
 Then, for every
 $R_0,U_0\in H^s(\R)$ and all external forces $H_r,H_u\in
 L^\infty(0,T_0;H^s(\R))$, the linear system \eqref{6.linest}
  possesses a unique solution $(R,U)\in C(0,T_0;H^s(\R))$ and the following estimate holds:
  \begin{multline}\label{6.esst}
  \|R(T)\|_{H^s}^2+\|U(T)\|_{H^s}^2\le Ce^{KT}\(\|R(0)\|_{H^s}^2+\|U(0)\|_{H^s}^2\)+\\+C\int_0^Te^{K(T-s)}
  \(\|H_r(s)\|_{H^s}^2+\|H_u(s)\|_{H^s}^2\)\, ds\,,
  \end{multline}
  where the constants $C$ and $K$ depend only on Sobolev norms of $r$ and $u$.
  \end{proposition}
\begin{proof} Introduce the energy density and flux
    \begin{equation}\label{E2-energy}
      \begin{array}{rcl}
        \widehat E_{{\rm lin}}(U,R) &=& \re^{2r}U^2 + 2 e^{4r}R^2\\[2mm]
        \widehat F_{{\rm lin}}(U,R) &=&
        2(u+k)\re^{2r}U^2
        +4(u+k)\re^{4r}R^2 + 4\re^{4r} U R\,.
      \end{array}
    \end{equation}
    Then a straightforward calculation, using (\ref{6.linest}), gives
    \[
    \partial_T\widehat E_{{\rm lin}} + \partial_X\widehat F_{{\rm lin}}
    =  - 4u_X\re^{2r}U^2 - 4u_X\re^{4r} R^2
    + 2\re^{2r}U H_u + 4\re^{4r}RH_r\,.
    \]
    Integrate over $\R$, and use vanishing of the flux $\widehat F_{{\rm lin}}$
    at infinity,
    \begin{equation}\label{7.good}
    \partial_TE_{{\rm lin}} =
    - \big(4u_X\re^{2r} , U^2 \big)_{L^2} - \big( 4u_X\re^{4r} ,  R^2 \big)_{L^2}
    + \big(2\re^{2r}U, H_u \big)_{L^2} + \big(4\re^{4r}R , H_r \big)_{L^2}\,,
    \end{equation}
    where
    \[
    E_{{\rm lin}} = \big( \widehat E_{{\rm lin}},1\big)_{L^2}\,.
    \]
    Now use the embedding of $H^1(\R)$ in the space of continuous functions
and the Cauchy-Schwarz inequality to arrive at
$$
\frac d{dT}E_{\rm lin}(T)\le K E_{\rm lin}(T)+C\(\|H_u(T)\|^2_{L^2}+\|H_r(T)\|^2_{L^2}\),
$$
where the constants $C$ and $K$ depend only on the norms of second derivatives of $u$ and $r$. The Gronwall
inequality now gives the desired estimate \eqref{6.esst} for $s=0$. For other values of $s$ it can be proved
analogously by differentiating \eqref{6.linest} in space sufficiently many times.
\par
{\color{black} Let us now discuss the case $s=1$ in more detail. Differentiating equations
\eqref{6.linest} in space, we see that the functions $\DX U$ and $\DX R$ satisfy the analogue of \eqref{6.linest}
with new right-hand sides
$$
H^1_r=\partial_X H_r-2U\partial_X^2r+2\partial_xu\partial_xR,\
\ H^1_u:=\partial_XH_u-2\DX(\DX u U)-4\DX(\DX re^{2r}R)\,.
$$
Applying the estimate \eqref{6.esst} with $s=0$  to the
equations for $\DX U$ and $\DX R$, we arrive at
\begin{multline}\label{6.esst1}
  \|\DX R(T)\|_{L^2}^2+\|\DX U(T)\|_{L^2}^2\le Ce^{KT}\(\|\DX R(0)\|_{L^2}^2+\|\DX U(0)\|_{L^2}^2\)+\\
  +C\int_0^Te^{K(T-s)}
  \(\|H_r^1(s)\|_{L^2}^2+\|H_u^1(s)\|_{L^2}^2\)\, ds.
  \end{multline}
Estimating the $L^2$-norms of $H^1_r$ and $H^1_u$ using \eqref{6.apr} with $s=1$ (which gives
in particular that $\partial_X^2u$ and $\partial_X^2r$ are bounded in $L^\infty$) and
 inserting the result into \eqref{6.esst1}, we get
\begin{multline}\label{6.esst2}
  \|R(T)\|_{H^1}^2+\|U(T)\|_{H^1}^2\le Ce^{KT}\(\|R(0)\|_{H^1}^2+\|U(0)\|_{H^1}^2\)+\\
  +C\int_0^Te^{K(T-s)}
  \(\|H_r(s)\|_{H^1}^2+\|H_u(s)\|_{H^1}^2\)\, ds+\\+C\int_0^Te^{K(T-s)}
  \(\|U(s)\|_{H^1}^2+\|R(s)\|_{H^1}^2\)\, ds,
  \end{multline}
where we have also used the estimate \eqref{6.esst} with $s=0$.
 Gronwall's inequality applied to \eqref{6.esst2} gives the desired
 estimate for $s=1$. The case $s>1$ can be done analogously.} Thus, the proposition is proved.
\end{proof}

We now turn to the non-linear case and state the main result of this section.
\vspace{.15cm}

\begin{theorem}\label{Th7.main} Let $s\ge2$ and let $u_0,r_0\in H^s(\R)$. Then there exists $T_0>0$ and
 $C$, depending only on the $H^s$-norms of the initial data, and a unique local solution $(r(T), u(T))$
 of problem \eqref{6.WMEL} on the time interval $T\in [0,T_0]$ satisfying the estimate
 \begin{equation}\label{7.est}
 \|r(T)\|_{H^s(\R)}+\|u(T)\|_{H^s(\R)}\le C\,,\quad \forall\ T\le T_0\,.
 \end{equation}
\end{theorem}

\begin{proof} We want to use the energy conservation law \eqref{energy-def}
to control the $L^2$-norm of the solution, namely,
 utilize the following obvious estimates
 \begin{equation}\label{7.positive}
 e^{-4\|r\|_{L^\infty}}\big(\|r\|^2_{L^2}+\|u\|^2_{L^2}\big)\le E(r,u)\le
 Ce^{4\|r\|_{L^\infty}}\big(\|r\|^2_{L^2}+\|u\|^2_{L^2}\big)\,.
 \end{equation}
 However, at the level of $s=0$ the $L^\infty$-norm is not under control, and so
 the energy identity \eqref{energy-def} is not enough to control these $L^2$-norms. So, we need
 higher energy estimates which we will get by differentiating equations \eqref{6.WMEL} in $X$
 and using the linearized energies \eqref{E2-energy} $E_{lin}(R,U)$ for quantities
  $R=\partial_X^s r$ and $U=\partial_X^s u$, $s=1,2,\cdots$, which satisfy the same analogue of
   estimate \eqref{7.positive}. As we will see below, $s=1$ is also not enough to close the estimates, so
we prove for simplicity the result for $s=2$ only (the first value of $s$
where the estimate can be closed in an elementary way).
The case $s>2$ is analogous.
 \par
Firstly, to control the $L^\infty$-norm of $r$ we use the obvious estimate
\begin{multline}
\big| \|r(T)\|_{L^\infty}-\|r_0\|_{L^\infty} \big| \le T_0\|r_T\|_{L^\infty}\le\\\le
T_0\big(\|u_X\|_{L^\infty}+2\|u\|_{L^\infty}\|r_X\|_{L^\infty}\big)\le CT_0\big(1+\|(r,u)\|^2_{H^2}\big)\,,
\end{multline}
for $T\le T_0$.
Thus, if we assume that
\begin{equation}\label{7.smallt}
T_0\le \frac1{C(1+\|(r(T),u(T))\|^2_{H^2})}\,,\qand T\le T_0\,,
\end{equation}
we will have
\begin{equation}\label{7.r}
\Big|\|r(T)\|_{L^\infty}- \|r_0\|_{L^\infty}\Big|\le 1\,,\quad \forall\ T\le T_0\,.
\end{equation}
In turn, this will allow us to replace $\|r(T)\|_{L^\infty}$ by $\|r_0\|_{L^\infty}$ in equalities like \eqref{7.positive} and
 this norm  is under control.
 \par

 Secondly, differentiate equations \eqref{6.WMEL} twice in $X$
 and let $R:=\partial_X^2r$ and $U:=\partial_X^2 u$.
 These functions satisfy linear
 equations
 \[
 \begin{array}{rcl}
   R_T &=& -2(u+k) R_X - U_X   -4u_X R - 2r_X U \\[2mm]
   U_T &=& -2(u+k) U_X -2\re^{2r} R_X -6u_X U  - 12 r_X \re^{2r} R
   -8 r_X^3 \re^{2r}\,.
 \end{array}
 \]
 Crucial for us is that the dependence on the second derivatives is {\it linear}. Comparison with
  the equations (\ref{6.linest}) shows that $H_r$
 and $H_u$ are
 \[
H_r = -4u_X R \qand H_u = -4u_XU - 8r_X \re^{2r}R - 8 r_X^3 \re^{2r}\,.
\]
This together with the embedding of $H^1(\R)$ in the space of
  continuous functions allows us to write the estimate
\begin{equation}\label{7.ext}
\|H_r(T)\|_{L^2}+\|H_u(T)\|_{L^2}\le Q\big(\|(r(T),u(T))\|_{H^2}\big)\,,
\end{equation}
for some monotone increasing function $Q$. Using now identity \eqref{7.good}
 and the key estimate \eqref{7.ext}, we arrive at
\begin{equation}\label{7.mest}
\frac d{dT}\(E_{\rm lin}(r_{XX}(T),u_{XX}(T))+E(r,u)\)\le Q_1\(\|(r(T),u(T))\|_{H^2}\)\,,
\end{equation}
with $E_{{\rm lin}}(r_{XX},u_{XX}):=(e^{2r},u_{XX}^2)_{L^2}+2(e^{4r},r_{XX}^2)_{L^2}$.
Moreover, using \eqref{7.positive} together with its
analogue for $E_{\rm lin}$ and assumption \eqref{7.r}, we get
\begin{multline}\label{7.pos1}
C^{-1}e^{-4\|r_0\|_{L^\infty}}\|u(T),r(T)\|^2_{H^2}\le
E_{\rm lin}(r_{XX}(T),u_{XX}(T))+E(r(T),u(T))\le\\\le Ce^{4\|r_0\|_{L^\infty}}\|u(T),r(T)\|^2_{H^2}
\end{multline}
for all $T\le T_0$ if the extra assumption \eqref{7.smallt} is satisfied.
Integrating now \eqref{7.mest} in time $T\le T_0$ and using the last estimate, we arrive at
\begin{equation}\label{7.fin}
\sup_{T\le T_0}\|(r(T),u(T)\|^2_{H^2}\le Q_2\big(\|(r_0,u_0)\|_{H^2}\big) +
T_0Q_3\(\sup_{T\le T_0}\|(r(T),u(T))\|_{H^2}\)
\end{equation}
for some monotone increasing functions $Q_2$ and $Q_3$. This estimate allows us to fix the time interval
$T_0=T_0(\|(r_0,u_0)\|_{H^2})$
 to be small enough   in order to get the desired estimate
 $$
 \|(r(T),u(T))\|_{H^2}\le 2Q_2(\|(r_0,u_0)\|_{H^2})\,,\quad \forall\ T\le T_0\,,
 $$
which coincides with \eqref{7.est} for $s=2$. Finally, shrinking further the lifespan $T_0$ of the solution
if necessary, we satisfy assumption \eqref{7.smallt} as well. This finishes the derivation of \eqref{7.est}
 for $s=2$.
 Existence and uniqueness of the solution
is straightforward when the proper a priori estimate is verified. Thus, the theorem is proved.
\end{proof}

\section{Approximate solutions of the perturbed WMEs}
\label{sec-approx-solns}
\setcounter{equation}{0}

In this section we study the finite-order Taylor expansion
in $\eb$ of $(\hat A,\hat \phi)$ in (\ref{0.ans}).
Substitution of (\ref{0.ans}) into (\ref{nls-1})
gives the exact equations
\begin{equation}\label{eq2.31}
\begin{array}{rcl}
  \displaystyle
  \partial_T\widetilde\phi + (k+\partial_X\widetilde\phi)^2-\gamma  \widetilde A^2
  +\omega-\eb^2\widetilde A^{-1}\partial_X^2\widetilde A  &=&0\\[2mm]
  \partial_T\widetilde A + 2(k+\partial_X\widetilde\phi)\partial_X\widetilde A
  +A\partial_X^2\widetilde\phi  &=&0\,.
\end{array}
\end{equation}
As before, we assume that $\gamma=\omega+k^2$ (the general case is reduced to this particular one by scaling).
These equations are expressed in terms of $\widetilde\phi$ as they
will be needed below for estimates on
the phase.  However, the estimates on the Taylor expansions will be
carried out in $(\widetilde r,\widetilde u)$ variables.  We transform the
equations in two steps.  First introduce the new variable $\widetilde r$
defined by $\widetilde A=e^{\widetilde r}$ to eliminate the singularity
at $\widetilde A=0$. Then \eqref{eq2.31} reads
\begin{equation}\label{eq2.32}
\begin{array}{rcl}
  \displaystyle
  \partial_T\widetilde\phi + (k+\partial_X\widetilde\phi)^2-\gamma  e^{2\widetilde r}
  +\gamma-k^2-\eb^2\partial_X^2\widetilde r-\eb^2(\partial_X\widetilde r)^2  &=&0\\[2mm]
  \partial_T\widetilde r + 2(k+\partial_X\widetilde \phi)\partial_X\widetilde r
  +\partial_X^2\widetilde\phi  &=&0\,.
\end{array}
\end{equation}
Differentiating the first equation in $X$ and inserting
$\widetilde u=\DX\widetilde \phi$, we end up with
\begin{equation}\label{eq2.33}
\begin{array}{rcl}
  \displaystyle
  \partial_T\widetilde u + \DX(k+\widetilde u)^2-\gamma  \DX e^{2\widetilde r}
  -\eb^2\partial_X^3\widetilde r-\eb^2\DX(\partial_X\widetilde r)^2  &=&0\\[2mm]
  \partial_T\widetilde r + 2(k+\widetilde u)\partial_X \widetilde r
  +\partial_X\widetilde u  &=&0\,,
\end{array}
\end{equation}
which is a perturbed version of SWEs \eqref{6.WMEL} studied earlier.
These equations (\ref{eq2.31})-(\ref{eq2.33}) are exact.  However, at this stage
we will prove the existence of $\eb^{2n}$-approximations $(\widehat r,\widehat\phi)$
to the exact solution $(\widetilde r,\widetilde\phi)$ only.
\par

We construct the approximate solution $(\hat r,\hat u)$ for equations \eqref{eq2.33}
and then lift the result to the initial equations \eqref{eq2.31}
to obtain phase information.
The approximate solution $(\hat r,\hat u)$ will be constructed in
the class of square integrable
  functions (in agreement with Theorem \ref{Th7.main}). However,
  the functions $\hat A$ and $\hat\phi$ are {\it not square integrable} in general, but will
  satisfy the following integrability properties:
  \begin{equation}
  \hat A-1, \DT\hat A,\DX\hat A,\DT\hat\phi,\DX\hat\phi\in H^s(\R)
  \end{equation}
for some $s>0$.
\par
Assume
 that the smooth local solution $(r(T),u(T))$ of the limit system \eqref{6.WMEL} is
  given and satisfies \eqref{7.est} for some $s\ge2$. The existence of such a solution
   is confirmed in Theorem \ref{Th7.main} for some small $T_0$ depending on the initial data,
    but in this section the lifespan of this given solution is not assumed to be small.
\par
Expand the approximate solution $(\hat r,\hat u)$
into a Taylor series in $\eb$:
\begin{equation}
\begin{array}{rcl}\label{taylor}
  \hat r(T,\eps) &:=& r(T)+\eb^2 r_1(T)+\eb^4 r_2(T)+\cdots+\eb^{2n}r_{n}(T),\\[2mm]
 \hat u(T,\eps) &:=& u(T)+\eb^2 u_1(T)+\eb^4 u_2(T)+\cdots+\eb^{2n}u_{n}(T)\,.
\end{array}
\end{equation}
Then, inserting these expansions into \eqref{eq2.33} and equating terms
with equal powers of $\eb$, we get the
recursive equations
\begin{equation}\label{6.lin}
\begin{cases}
  \partial_T r_l=-\partial_X u_l-2(u+k)\partial_X r_l-2\partial_X r u_l+H_{l-1,r}(T)\,,
  \quad r_l\big|_{T=0}=0,\\
  \partial_T u_l =-2\partial_X((u+k)u_l)-2\partial_X(e^{2r}r_l)+H_{l-1,u}(T)\,,
  \quad u_l\big|_{T=0}=0\,,
\end{cases}
\end{equation}
where $l\ge1$ and $(r_0,u_0)=(r,u)$.
The smooth functions $H_{l-1,r}$ and $H_{l-1,u}$ depend only on
$r,r_1,\cdots r_{l-1}$ and $u,u_1,\cdots, u_{l-1}$
and their derivatives up to order $3$, and so the sequence of linear equations \eqref{6.lin}
can be solved recursively. Moreover, the residual
$\widehat \Res_{n}(T)$ satisfies
$$
\widehat \Res_{n,\star}=\eb^{2(n+1)}R_{n,\star}(\eb, D u,Du_1,\cdots, Du_n, Dr,Dr_1,\cdots, D r_n)
$$
for some smooth functions $R_{n,\star}$.
Here ``$\star$'' represents "$r$" or "$u$" and $Du$
means the collection of all
$X$-derivatives of $u$ up to order  $3$.
For instance,
\begin{equation}
H_{0,r}=0,\ \ H_{0,u}=\partial_X^3r+\DX(\DX r)^2
\end{equation}
and
\begin{multline}
R_{1,r}=2u_1\DX r_1,\ \ R_{1,u}=2u_1\DX u_1-\gamma\eb^{-4}\DX(e^{2r}(e^{2\eb^2 r_1}-1-2\eb^2 r_1))-\\-
\DX^3r_1-2\DX(\DX r\DX r_1)-\eb^2\DX(\DX r_1)^2.
\end{multline}
Thus, in order to get the desired result for the approximate solutions $(\hat r,\hat u)$, we just need
to get good estimates for the $H^s$-norms of the solutions for
the linear problems \eqref{6.lin}. This is possible due to Proposition \ref{Prop6.main}
and gives the following result.
 \vspace{.15cm}
\begin{proposition}\label{Prop7.RUa} Let $s\ge0$ and $n\in\Bbb N$ be fixed and let $(r(T),u(T))$,
$t\le T_0$ be a solution of the limit problem \eqref{6.WMEL} satisfying
\begin{equation}\label{in-good}
\|r(t)\|_{H^{s+3n+5}(\R)}+\|u(T)\|_{H^{s+3n+5}(\R)}\le C.
\end{equation}
Then there exists an approximate solution $(\widehat r,\widehat u)$  which satisfies
 \begin{equation}\label{5.WUPa}
 \begin{cases}
 \partial_T\hat r=-\partial_X \hat u-2(\hat u+k)\partial_X\hat r+\widehat\Res_{n,r}(T),\\[2mm]
 \partial_T\hat u=\eb^2\partial_X^3\hat r-\partial_X((\hat u+k))^2+\eb^2\partial_X(\partial_X \hat r)^2+ \gamma\partial_X(e^{2\hat r})+ \widehat\Res_{n,u}(T)\,,
 \end{cases}
 \end{equation}
 defined on the same time interval $T\in[0,T_0]$ such that
\begin{equation}
\|\hat r(T)\|_{H^{s+5}}+\|\hat u(T)\|_{H^{s+5}}+\|\DT\hat r(T)\|_{H^{s+2}} \le C_1\,,
\end{equation}
and
\begin{equation}
\|\widehat \Res_{n,r}(T)\|_{H^{s+2}}+\|\widehat \Res_{n,u}(T)\|_{H^{s+2}}\le C_1\eb^{2(n+1)}\,,
\end{equation}
for all $T\in[0,T_0]$, and for some constant $C_1$ which
is independent of $\eb$.
\end{proposition}

\subsection{Lifting to the phase equation}
\label{subsec-phase}

Let us now lift the approximate solution $(\hat r,\hat u)$
to the phase variables in equation \eqref{eq2.31}. The
principal difficulty here is defining the approximate phase $\hat\phi$
(going from \eqref{eq2.33} to \eqref{eq2.32}), as the lifting from equations
\eqref{eq2.32} to \eqref{eq2.31} is immediate.

For exact solutions we have the relation
$u=\DX\phi$, so the obvious way to define the approximate phase would
be to integrate over the slow space variable:
 $$
 \hat\phi(T,X)=\int_0^X\hat u(T,s)\,ds+\bar \phi(T)\,;
 $$
 this strategy is used in \cite{ds09}.  However, in this case it is not
 clear how to determine the function $\bar\phi(T)$,
and so a complete picture is lacking.
We propose an alternative way, namely, to use the first
 equation of \eqref{eq2.32} to \emph{solve} for the phase,
\begin{equation}\label{phi-def}
\partial_T\hat\phi:=- (k+\hat u)^2+\gamma  e^{2\hat r}
  -\gamma+k^2+\eb^2\partial_X^2\hat r+\eb^2(\partial_X\hat r)^2,\ \ \hat\phi|_{T=0}=\phi_0\,.
\end{equation}
Then, since $(\hat r,\hat u)$ is already defined, the phase $\hat\phi(T)$ will
 be restored in a unique way. From this formula we see that indeed $\DT\hat\phi(T)\in H^{s+3}(\R)$
  for $T\le T_0$. Moreover, comparing this definition with the second equation of \eqref{5.WUPa},
  we see that
$$
\DT(\DX\hat\phi-\hat u)=-\widehat\Res_{n,u}(T),\ \ \DX\hat\phi\big|_{T=0}-\hat u\big|_{T=0}=0
$$
and therefore
$$
\DX\hat\phi(T)=\hat u(T)-\int_0^T\widehat\Res_{n,u}(s)\,ds.
$$
Inserting this identity into the right-hand side of equation \eqref{phi-def} and to the first
equation of \eqref{5.WUPa}, we conclude that the pair $(\hat r,\hat\phi)$ solves
\begin{equation}\label{eq2.34}
\begin{array}{rcl}
  \displaystyle
  \partial_T\hat\phi + (k+\partial_X\hat\phi)^2-\gamma  e^{2\hat r}
  +\gamma-k^2-\eb^2\partial_X^2\hat r-\eb^2(\partial_X\hat r)^2  &=&\Res_{n,\phi}(T)\\[2mm]
  \partial_T\hat r + 2(k+\partial_X\hat\phi)\partial_X \hat r
  +\partial_X^2\hat\phi  &=&\Res_{n,r}(T)\,,
\end{array}
\end{equation}
where
\begin{equation}\label{eq.res-phi}
\Res_{n,\phi}(T)=-\(2k+2\hat u(T)-\int_0^T\widehat\Res_{n,u}(s)\,ds\)\int_0^T\widehat\Res_{n,u}(s)\,ds
\end{equation}
and
\begin{equation}\label{eq.res-r}
\Res_{n,r}(T)=\widehat\Res_{n,r}(T)-\int_0^T\DX\widehat\Res_{n,u}(s)\,ds-
2\DX\hat r\int_0^T\widehat\Res_{n,u}(s)\,ds.
\end{equation}
Thus, we have proved the following result which can be considered as the main result of this section.
\vspace{.15cm}

\begin{theorem}\label{Th6.main} Let the assumptions of Proposition \ref{Prop7.RUa} hold and suppose $(\hat r, \hat u)$ is the
approximate solution constructed there. Then the function $\hat A:=e^{\hat r}$ and $\hat\phi$
 defined by \eqref{phi-def} satisfy
\begin{equation}\label{eq2.35}
\begin{array}{rcl}
  \displaystyle
  \partial_T\widehat\phi + (k+\partial_X\widehat\phi)^2-\gamma  \widehat A^2
  +\gamma-k^2 -\eb^2\widehat A^{-1}\partial_X^2\widehat A  &=&\Res_{n,\phi}\\[2mm]
  \partial_T\widehat A + 2(k+\partial_X\widehat \phi)\partial_X \widehat A
  +A\partial_X^2\widehat \phi  &=&\Res_{n,A}\,,
\end{array}
\end{equation}
where the residuals $\Res_{n,\phi}$ and $\Res_{n,r}$ are defined by \eqref{eq.res-phi}
 and \eqref{eq.res-r} respectively, the residual $\Res_{n,A}:=e^{\hat r}\Res_{n,r}$ and
  the following estimate holds:
\begin{equation}\label{4.rres}
 \|\Res_{n,A}(T)\|_{H^{s+1}}+\|\Res_{n,\phi}(T)\|_{H^{s+1}}\le C\eb^{2(n+1)}\,,
\end{equation}
 for $T\le T_0$. The approximate solution $(\hat A,\hat \phi)=(e^{\widehat r},\widehat\phi)$
 also satisfies the estimate
\begin{equation}\label{4.phi}
\|\widehat r(T)\|_{H^{s+5}}+\|\DX\hat\phi(T)\|_{H^{s+2}}+
\|\DT\hat r(T)\|_{H^{s+2}} \le C
\end{equation}
for $T\le T_0$ and some constant $C$ which is independent of $\eb$.
\end{theorem}
\vspace{.15cm}

\begin{remark} Restoring the phase $\hat \phi$ via equation \eqref{phi-def}
  gives more detailed information about the phase then one might expect.
  First of all, with this approach, the phase
  is restored in a unique way which, in turn, leads to the construction
  of an approximate solution of
the initial NLS also in a unique way and this is crucial for our proof of validity, see next section below.
\par
 In addition, it follows from equation \eqref{phi-def} that $\DT\hat\phi\in L^2$ and, therefore,
 $\hat\phi(T)-\hat\phi(0)\in L^2$ which gives a useful extra information about the approximate phase.
 \par
 Finally, our approach allows to verify that the difference $\widetilde \phi-\widehat\phi$ between the exact
  and approximate phases is in $L^2$ and is controlled by the appropriate power of $\eb$ (although the details of such
  estimates are outside the scope of this paper).
\end{remark}

\section{Exact solutions in Sobolev spaces}
  \label{sec-exactsolutions}
  \setcounter{equation}{0}

  In this section the properties of the exact
  solution (\ref{exact-solution-nls}) needed to complete the proof of
  Theorem \ref{Th0.main} are established.  It is
  assumed that the two estimates (\ref{4.rres})
  and (\ref{4.phi}) hold with $s=0$.

  Express the exact solution in the form
\begin{equation}\label{Psi-A-decomposition}
  \widetilde \Psi(T,X,\eps)=\widehat\Psi(T,X,\eps)V(T,X,\eps)\,,\quad
  V = 1+ W_1+\ri W_2\,,
\end{equation}
where $\widehat\Psi$ is the approximation of (\ref{0.ans})
to $n-$th order obtained
in \S\ref{sec-approx-solns} and
$W_1+\ri W_2$ is an unknown complex-valued function. Substitution of
this expression into the cubic NLS equation (\ref{nls-1}) gives
the following equation:
{\color{black}
$$
\ri\partial_T V+\eb\partial_X^2 V+2\eb\frac{\DX \widehat\Psi}{\widehat\Psi}\DX V+
\frac{\ri\partial_T \widehat\Psi+
\eb\partial_X^2\widehat\Psi}{\widehat \Psi}V+\gamma\eb^{-1}|\widehat\Psi|^2|V|^2V=0.
$$
Recall that, according to \eqref{0.ans},
  \[
  \widehat  \Psi(T,X,\eps) = \re^{\hat r}
  \re^{\ri\eb^{-1}\widehat\Theta}\,, \ \widehat\Theta:=\omega T+kX+\hat\phi\,,
  \]
  where $W_1,W_2,\widehat\Theta,\widehat r$, and $\widehat \phi$ are all functions
  of $(T,X,\eps)$, we obtain that
  $$
\frac{\ri\partial_T \widehat\Psi+
\eb\partial_X^2\widehat\Psi}{\widehat \Psi}V=
i\hat A^{-1}\(\DT \hat A+\hat A\partial_X^2\hat\Theta+2\DX \hat A\DX\hat\Theta\)-
\eb^{-1}\ri\(\DT\hat\Theta+(\DX\hat\Theta)^2-\hat A^{-1}\partial_X^2 \hat A\).
$$
Using now that the functions $\hat A$ and $\hat\phi$ satisfy equation \eqref{eq2.35} together with the dispersion relation
$\gamma=\omega+k^2$, we arrive at the following equation for $V$:}
\begin{equation}\label{4.main}
  \begin{array}{rcl}
&&
\ri\partial_T V+\eb\partial_X^2 V+2\ri (k+\partial_X\hat\phi)\partial_X V
+2\eb \partial_X\hat r\partial_X V+\\[2mm]
&&\hspace{2.5cm} +
\gamma e^{2\hat r}\eb^{-1}V(|V|^2-1) +(\ri e^{-\hat r}\Res_A-\eb^{-1}\Res_\phi)V=0\,.
  \end{array}
  \end{equation}
Separating real and imaginary parts, \eqref{4.main} splits into
  an equation for $W_1$ and $W_2$,
\begin{equation}\label{4.real}
\begin{cases}
\partial_T W_1=-\eb(\partial_X^2+2\partial_X\hat r\partial_X)W_2-2(k+\partial_X\hat
\phi)\partial_X W_1+\\[2mm]\hspace{4.0cm}
+\Ree\{\Cal H(T)(1+W)\}+\eb^{-1}F_1(W)\\[2mm]
\partial_T W_2=\eb(\partial_X^2+2\partial_X\hat r\partial_X+2\eb^{-2}\gamma e^{2\hat r})W_1-2(k+\partial_X\hat
\phi)\partial_X W_2+\\[2mm]\hspace{4.0cm} +\Imm\{\Cal H(T)(1+W)\}+\eb^{-1}F_2(W),
\end{cases}
\end{equation}
where the non-linearity $F$ is a polynomial containing
only quadratic and cubic terms and
\begin{equation}\label{H-Res-def}
 \Cal H:=e^{-\hat r}\Res_A+\ri \eb^{-1}\Res_\phi.
\end{equation}
Thus, in these new coordinates the functions $W_1$ and $W_2$ are responsible for the deviation of the exact
solution $\widetilde\Psi(T,X)$ from the approximate solution $\widehat \Psi(T,X)$ constructed via the higher order WMT,
so our task here is to verify that these functions are in a sense small.
\par
  In order to solve the semilinear
 equation (\ref{4.real}) on the interval $T\le T_0$ it is enough to get good estimates in the $H^s$-norm
  for the linearized non-homogeneous equation:
\begin{equation}\label{4.lin}
\begin{cases}
\partial_T W_1=-\eb(\partial_X^2+2\partial_X\hat r\partial_X)W_2-2(k+\partial_X\hat
\phi)\partial_X W_1+H_1(T)\\[2mm]
\partial_T W_2=\eb(\partial_X^2+2\partial_X\hat r\partial_X+2\eb^{-2}\gamma e^{2\hat r})W_1-2(k+\partial_X\hat\phi)\partial_X W_2+H_2(T),
\end{cases}
\end{equation}
where $H_i(T)$ are given external forces. This is done in the following theorem.
\vspace{.15cm}

\begin{theorem}\label{Th4.lin} Let the above conditions on $\hat r$ and $\hat\phi$ hold and
 let $\gamma=-1$. Then,
for every $W(0)\in H^1(\R)$, {\color{black} there exists a unique   solution
$W\in C([0,T_0],H^1)$ for any $T_0\in\R_+$  of
 \eqref{4.lin} and this solution} satisfies the following estimate:
 \begin{multline}\label{4.linest}
   \|\partial_X W(T)\|_{L^2}^2+\eb^{-2}\|W_1(T)\|_{L^2}^2+\|W_2(T)\|^2_{L^2}\le\\[2mm]\le
    C(\|\partial_X W(0)\|_{L^2}^2+\eb^{-2}\|W_1(0)\|^2_{L^2}+\|W_2(0)\|^2_{L^2})e^{C T}+\\[2mm]
   +C\int_0^Te^{C(T-\tau)}\Big(\|\partial_X H(\tau)\|_{L^2}^2+\eb^{-2}\|H_1(\tau)\|^2_{L^2}+\|H_2(\tau)\|^2_{L^2}\Big)\,d\tau\,,
 \end{multline}
 where the constant $C$ is independent of $\eb$.
\end{theorem}
\vspace{.15cm}

\begin{proof}{\color{black} We restrict attention to a formal verification
    of the key estimate \eqref{4.linest}.
    The existence of a solution can then be confirmed
    {\it a posteriori} in a standard way using an appropriate
   approximation scheme.}
  \par
      Fix $\gamma=-1$, multiply the
 first and the second equations in \eqref{4.lin} by
 $-\partial_X(e^{2\hat r}\partial_X W_1)+2\eb^{-2}e^{4\hat{r}}W_1$ and
  $-\partial_X(e^{2\hat r} \partial_XW_2)$,
 integrate over $X$, take a sum, and use the following obvious identity:
$$
\partial_X^2W+2\hat r_X\partial_XW=e^{-2\hat r}\partial_X(e^{2\hat r}\partial_XW)\,.
$$
After cancellation of the leading terms in the right-hand side and integration by parts,
   this gives
 \begin{multline}\label{4.gr}
\frac12\frac d{dT}\Big( \|e^{\hat r}\partial_X W\|^2_{L^2}+2\eb^{-2}\|e^{2\hat {r}}W_1\|^2_{L^2}\Big)-
\Big(\partial_T\hat r, (e^{\hat r}\partial_XW)^2\Big) -\\[1mm]
- 4\eb^{-2} \big(\partial_T\hat r, (e^{2\hat r}W_1)^2\big)=
 -\(e^{2\hat r}\partial_X(e^{-2\hat r}(k+\partial_X\hat\phi)),|e^{\hat r}\partial_XW|^2\)+\\[2mm]
 +2\eb^{-2}\big(e^{-4\hat r}(e^{4\hat r}(k+\hat\phi_X))_X ,|e^{2\hat{r}}W_1|^2\big)+
 \big(H_1(T),-\partial_X(e^{2\hat r}\partial_X W_1)
 +2\eb^{-2}e^{4\hat{r}}W_1\big)+\\[2mm]
 +\big(H_2(T),-\partial_X(e^{2\hat r} \partial_XW_2)\big)\,.
 \end{multline}
Using \eqref{4.phi} with $s=0$, we have
$$
\|\partial_T\hat r\|_{L^\infty}+\|\hat r\|_{W^{1,\infty}}+\|\DX\hat\phi\|_{W^{1,\infty}}\le C\,,
$$
and therefore \eqref{4.gr} can be transformed to
 \begin{multline}\label{4.gra}
   \frac d{dT}\Big(\|e^{\hat r}\partial_X W\|^2_{L^2}+2\eb^{-2}\|e^{2\hat r}W_1\|^2_{L^2}\Big)-\\[2mm]
   - C\Big(\|e^{\hat r}\partial_X W\|^2_{L^2}+2\eb^{-2}\|e^{2\hat r}W_1\|^2_{L^2}\Big)\le C\Big(\|\partial_X H(T)\|^2_{L^2}+\eb^{-2}\|H_1(T)\|^2_{L^2}\Big)\,.
 \end{multline}
 The Gronwall inequality then gives the following estimate:
\begin{multline}\label{4.gr1}
\|\partial_X W\|^2_{L^2}+2\eb^{-2}\|W_1\|^2_{L^2}\le
C\big(\|\partial_X W(0)\|^2_{L^2}+\eb^{-2}\|W_1(0)\|^2_{L^2}\big)e^{C T}
+\\[2mm]
+ C\int_0^Te^{C(T-\tau)}\Big(\|\partial_XH(\tau)\|^2_{L^2}+\eb^{-2}\|H_1(\tau)\|^2_{L^2}\Big)\,d\tau
 \end{multline}
which coincides with \eqref{4.linest}
up to the $L^2$-norm of $W_2$ which we
 now need to estimate. To this end, multiply the second equation of \eqref{4.lin} by $W_2$ and
 integrate over $X$. {\color{black} This gives
\begin{multline*}
\frac12\frac d{dt}\|W_2\|^2_{L^2}=-\eb(\DX W_1,\DX W_2)+2\eb(\DX \hat r\DX W_1,W_2)+\\+
2\gamma\eb^{-1}e^{2\hat r}(W_1,W_2)-2((k+\DX\hat\phi)\DX W_2,W_2)+(H_2,W_2).
\end{multline*}
Using again \eqref{4.phi} together with Cauchy-Schwarz inequality,}
 we get
  \[
\frac12\frac d{dT}\|W_2\|^2_{L^2}- C\|W_2\|^2_{L^2}\le
C\Big(\eb\|\partial_X W\|^2_{L^2}+\eb^{-2}\|W_1\|^2_{L^2}+\|H_2(T)\|^2_{L^2}\Big)\,,
\]
which together with \eqref{4.gr1} gives
\begin{multline}\label{4.gr2}
\|W_2(T)\|_{L^2}^2\le
 C\Big(\|\partial_X W(0)\|^2_{L^2}+\eb^{-2}\|W_1(0)\|^2_{L^2}+\|W_2(0)\|^2_{L^2}\Big)e^{CT}+\\[2mm]+
 C\int_0^Te^{C(T-s)}\Big(\|\partial_X H(s)\|^2_{L^2}+\eb^{-2}\|H_1(s)\|^2_{L^2}+\|H_2(s)\|^2_{L^2}\Big)\,ds\,,
\end{multline}
which gives the desired estimate \eqref{4.linest} and finishes the proof of the theorem.
\end{proof}
\begin{remark}
  {\color{black} We  estimate the solution $W$ of problem \eqref{4.lin} in the space $H^1$ only,
  which is sufficient for the present purposes.
   The analogue of this estimate in $H^s$ with  $s>1$  can be verified
  analogously to the proof of Proposition \ref{Prop6.main}
  by differentiating equation \eqref{4.lin} in $X$ sufficiently many times.}
\end{remark}

  We are now ready to return to the non-linear equation \eqref{4.real} with zero initial data.  For this we need $\mathcal{H}$ in (\ref{H-Res-def}) to
  satisfy
$$
\Cal H(T)\sim O(\eb^{2n+1})\,,
$$
and this follows from the estimate
(\ref{4.rres}) with $s=0$. Now, inverting the linear part
and using \eqref{4.linest} together with the fact that $H^1$
is an algebra we arrive at
 \begin{multline}\label{4.grr}
\|W(T)\|_{H^1}^2\le C\eb^{-2}\sup_{T\le T_0}\Big(\|\mathcal H(T)(1+W(T))\|_{H^1}^2+
\eb^{\color{black}-2}\|F(W(T))\|^2_{H^1}\Big)\le\\[2mm] \le C\eb^{4n}+C\eb^{4n}\sup_{T\le T_0}\|W(T)\|^2_{H^1}+\\+
C\eb^{\color{black}-4}\sup_{T\le T_0}\|W(T)\|^4_{H^1}\(1+\sup_{T\le T_0}\|W(T)\|^2_{H^1}\).
 \end{multline}
{\color{black} Let $\widetilde W:=\eb^{-2n}W$. Then, the new function satisfies
$$
\|\widetilde W(T)\|_{H^1}^2   \le C+C\eb^{4n}\sup_{T\le T_0}\|\widetilde W(T)\|^2_{H^1}+
C\eb^{8n-4}\sup_{T\le T_0}\|\widetilde W(T)\|^4_{H^1}
\(1+\eb^{4n}\sup_{T\le T_0}\|\widetilde W(T)\|^2_{H^1}\)
$$
and we see that in the case $n\ge1$, the obtained estimate guarantees that $\|\widetilde W(T)\|_{H^1}\le 2C$
for all $T\in[0,T_0]$ for some $T_0>0$ which is independent of $\eb\to0$.
}
Thus, we have proved the following result.
\vspace{.15cm}

\begin{corollary}\label{Cor4.main} Let the above assumptions hold and suppose $n\ge1$. Then, for $\eb$ small enough,
  equation \eqref{4.real}
  with zero initial data possesses a unique solution $W(T)$ on
  the time interval $T\le T_0$,
 and the following estimate holds:
 \begin{equation}\label{4.good}
\sup_{T\le T_0}\|W(T)\|_{H^1(\R)}\le C\eb^{2n}\,.
 \end{equation}
 The constant $C$ depends on the norms of approximating solution, but is independent of $\eb$.
\end{corollary}

\noindent Estimate \eqref{4.good} is a standard corollary of \eqref{4.grr} and the existence of a
solution can be obtained using the semigroup technique since
\eqref{4.real} has a semilinear structure.
\par

This corollary gives the required estimate for the distance
between exact and approximate solutions of the NLS, {\color{black}
  recalling that, by definition,
$$
W(T)=(\widetilde\Psi(T)-\widehat\Psi(T))e^{-\ri\eb^{-1}\widehat \Theta(T)}\,,
$$
}
thereby completing the proof of Theorem \ref{Th0.main}.

\section{Linear stability of wavetrains and validity}
\label{sec-stability-and-validity}
\setcounter{equation}{0}

{\color{black} A surprise in Theorem \ref{Th0.main} is that the
  power of  $\eb$ in the estimate (\ref{est0-main}) is
  $\eb^{2n}$ rather than the more natural $\eb^{2n+1}$.  In this
  section we show why $\eb^{2n}$ is essential and can not be
  improved.  It is intimately connected with an algebraic
  instability of the wavetrain (\ref{nls-wavetrain}).
  \par
  Firstly, in contrast to the Gevrey
  spaces approach developed in \cite{ds09}, the spectral stability
  of the background wavetrain is a {\it necessary}
  condition for the validity of Whitham approximations in Sobolev spaces.
  For this reason,
it looks natural to use the energy norms related with this spectral stability in the proof of
 validity.  Actually, our presentation of the exact solution in the form of
  \eqref{Psi-A-decomposition} and the choice of the energy norm for equation \eqref{4.lin} are
   inspired by the spectral stability arguments.
   \par
   The linearization of the initial NLS (\ref{nls-1})
   on a wavetrain solution $e^{\ri kx+\ri\omega t}$ reads
$$
\ri\partial_tZ+\partial_x^2 Z+\gamma(2Z+e^{2(\ri kx+\ri\omega t)}\bar Z)=0,\ \ \gamma=\omega+k^2\,.
$$
In order to make this equation autonomous, let $Z=e^{\ri kx+\ri\omega t}W$, giving
 \begin{equation}\label{linear}
\Dt W-\ri\partial_x^2 W+2k\partial_x W-\ri \gamma(W+\bar W)=0\,
 \end{equation}
or, separating, real and imaginary parts,
\begin{equation}\label{llinear}
\begin{cases}
\Dt W_1=-\partial_x^2W_2-2k\partial_x W_1,\\
\Dt W_2=\partial_x^2 W_1-2k\partial_x W_2+2\gamma W_1
\end{cases}
\end{equation}
which coincides with the unscaled version ($\eb=1$) of equation \eqref{4.lin} with
$\hat r=\hat \phi=0$ and zero right-hand sides.

This equation possesses the conservation law
\begin{equation}\label{5.cons}
\frac d{dt}\(\|\partial_x W\|^2_{L^2}-2\gamma\|\operatorname{Re}W\|^2_{L^2}\)=0\,,
\end{equation}
which gives the spectral stability of the wavetrain for $\gamma<0$ and also is a prototype for
 our key estimate \eqref{4.linest}.
\par
Note also that \eqref{5.cons} gives the stability of the linearized equation in the space
\[
(W_1,W_2)\in H^1\times \dot H^1\,,
\]
where $\dot H^1$ is the corresponding
homogeneous Sobolev space, which is enough for spectral stability, but does not imply
the stability in a natural phase space $H^1\times H^1$ or $L^2\times L^2$.
\par
Actually, a bit more accurate analysis shows that problem \eqref{linear} is {\it unstable} in
 a natural space $H^1\times H^1$, namely, the $L^2$-norm of $W_2(t)$ may grow (at most linearly in time)
 despite the fact that $\|\partial_x W_2(t)\|$ remains bounded. As usual, such an instability is caused by a Jordan cell.
 Indeed, after the Fourier transform in $x$, \eqref{linear} reads
 $$
 \frac d{dt}\(\begin{matrix} \Cal F W_1\\\Cal F W_2\end{matrix}\)=\Bbb M(\xi)
 \(\begin{matrix} \Cal F W_1\\\Cal F W_2\end{matrix}\),\ \ \
 \Bbb M(\xi):=\(\begin{matrix}-2k\ri\xi,&\xi^2\\-\xi^2-\gamma,&-2k\ri\xi\end{matrix}\)
$$
and we see the Jordan cell in $\Bbb M(\xi)$ at $\xi=0$ which is responsible for the long-wave instability
and produces the linear growth of the $L^2$-norm of  $W_2(t)$ in time.
\par
In the scaled variables this linear growth produces an extra  factor $\eb^{-1}T\le \eb^{-1}T_0$
in the corresponding estimates. In particular, the factor $\eb^{-2}$ in the middle part of \eqref{4.grr}
is caused exactly by this instability. In turn, this decreases the power of $\eb$ in the right-hand
side of Theorem \ref{Th0.main} from the expected $\eb^{2n+1}$ to $\eb^{2n}$
and requires the use of the higher order Whitham approximations
in \S\ref{sec-approx-solns}.
}

\section{Concluding remarks}
\label{sec-cr}
\setcounter{equation}{0}

The main result of the paper is
validity of the Whitham theory as an
approximation to solutions that modulate a periodic travelling
wave of the NLS equation.  However, the theory does not rely
on the structure of NLS and so will carry over to other equations,
particularly non-integrable nonlinear wave equations, when
the Whitham equations are strictly hyperbolic, and the
appropriate energy functionals can be constructed.  As shown
  in \S\ref{sec-stability-and-validity}, replacing $\eb^{2n}$ with
  $\eb^{2(n+1)}$ on the right-hand side of (\ref{est0-main}) in
Theorem \ref{Th0.main} is unavoidable if we are working in $H^1$. However,
 there is still a room for refining the estimates utilizing homogeneous Sobolev spaces where this instability disappears.

An extension of interest is validity of multiphase WMT as an approximation
to modulation of multiphase periodic travelling waves of the
coupled nonlinear Schr\"odinger equation.  A proof of validity
in Gevrey spaces for this case has been given in \cite{bks20}.  The problem
with multiphase WMEs is that the characteristics can be
elliptic, hyperbolic, or mixed \cite{br19}.  Gevrey spaces are
indifferent to characteristic type, but the only hope for a proof
of validity in Sobolev spaces for the multiphase WMEs for coupled NLS
is to restrict to the case where all characteristics are hyperbolic,
and do not change type with time.

\subsection*{Acknowledgements}

The research of AK and SZ is partially supported by EPSRC Grant EP/P024920/1
on {\it Finite-dimensional reduction, inertial manifolds, and homoclinic structures in dissipative PDEs} and
the work of SZ is also supported by the grant  19-71-30004   of Russian Science Foundation.

\bibliographystyle{amsplain}

\begin{thebibliography}{99}

\bibitem{bnr14}
  \textsc{S. Benzoni-Gavage, P. Noble, \& L.M. Rodrigues}.
         {\it Slow modulations of periodic waves in Hamiltonian PDEs, with
           application to capillary fluids},
         J. Nonl. Sci. {\bf 24} 711--768 (2014).

\bibitem{tjb17}
\textsc{T.J. Bridges}.
{\it Symmetry, Phase Modulation, and Nonlinear Waves},
Cambridge University Press: Cambridge (2017).

\bibitem{bks20}
  \textsc{T.J. Bridges, A. Kostianko, \& G. Schneider}.
         {\it A proof of validity for multiphase Whitham modulation theory},
         Proc. Roy. Soc. Lond. A {\bf 476} 2020203 (2020).

       \bibitem{br19}
         \textsc{T.J. Bridges \& D.J. Ratliff}.
                {\it Krein signature and Whitham modulation theory: the sign of characteristics and the “sign characteristic”},
                Stud. Appl. Math. {\bf 142} 314--355 (2019).

              \bibitem{bhj16}
                \textsc{J.C. Bronski, V. Hur, \& M.A. Johnson}.
                       {\it Modulational instability in equations of KdV type},
                       in {\it New Approaches to Nonlinear Waves}, Edited
                       by \textsc{E. Tobisch},
Lect. Notes. Phys. {\bf 908} 83--133 (2016).

              \bibitem{bj10}
                \textsc{J.C. Bronski \& M.A. Johnson}.
                {\it  The modulational instability for a generalized Korteweg-de
                  Vries equation}, Arch. Rat. Mech. Anal. {\bf 197} 357--400 (2010).

         \bibitem{carles-book}
  \textsc{R. Carles}.
         {\it Semi-Classical Analysis for Nonlinear Schr\"odinger Equations},
         World Scientific: Singapore (2008).

\bibitem{ds09}
\textsc{W.-P. D\"ull \& G. Schneider}.
{\it Validity of Whitham's equations for the modulation of periodic traveling waves in the NLS equation}, J. Nonl. Sci. {\bf 19} 453--466 (2009).

\bibitem{Ger}
\textsc{P. G\'erard}. {\it Remarques sur l'analyse semi-classique de l'\'equation de Schr\"odinger non lin\'eaire},
S\'eminaire EDP de l'Ecole Polytechnique, Palaiseau, France (1992--93), lecture no. XIII.
\bibitem{gk14}
  \textsc{B. Grebert \& T. Kappeler}.
         {\it The defocusing NLS equation and its normal form},
         European Mathematical Society, Lecture Notes Series {\bf 18},
         DOI 10.4171/131 (2014).
\bibitem{Gren}
  \textsc{E. Grenier}.
         {\it Semiclassical limit of the nonlinear Schr\"odinger equation in
small time}, Proc. Amer. Math. Soc. {\bf 126} 523--530 (1998).
\bibitem{jlm99}
  \textsc{S. Jin, C.D. Levermore, \& D.W. McLaughlin}.
         {\it The semiclassical limit of the defocusing NLS hierarchy},
         Comm. Pure Appl. Math. {\bf 52} 613--654 (1999).

\bibitem{k00}
\textsc{A.M. Kamchatnov}.
{\it Nonlinear Periodic Waves and Their Modulations},
World Scientific: Singapore (2000).
\bibitem{Majda}
  \textsc{A. Majda}.
         {\it Compressible fluid flow and systems of conservation laws in
 several space variables}, Appl. Math. Sci {\bf 53}, Springer: Berlin (1984).

         \bibitem{su17}
\textsc{G. Schneider \& H. Uecker}.
       {\it Nonlinear PDEs: A Dynamical Systems Approach},
       American Mathematical Society: Providence (2017).

       \bibitem{w74}
\textsc{G.B. Whitham}. {\it Linear and Nonlinear Waves},
Wiley-Interscience: New York (1974).

\end{thebibliography}

\end{document}